\documentclass[11pt,leqno]{amsart}

\newtheorem{theorem}{\sc Theorem}[section]
\newtheorem{lem}[theorem]{\sc Lemma}
\newtheorem{prop}[theorem]{\sc Proposition}

\begin{document}
\author{Eloisa Detomi}
\address{Dipartimento di Matematica Pura e Applicata, Universit\`a di Padova,
 Via Trieste 63, 35121 Padova , Italy}
\email{detomi@math.unipd.it}
\author{Marta Morigi}
\address{Dipartimento di Matematica, Universit\`a di Bologna,
Piazza di Porta San Donato 5, 40126 Bologna, Italy}
\email{mmorigi@dm.unibo.it}
\author{Pavel Shumyatsky}
\address{Department of Mathematics, University of Brasilia,
Brasilia-DF, 70910-900 Brazil ,pavel@unb.br}
\email{pavel@mat.unb.br}

\title{Bounding the Exponent of a Verbal Subgroup}
\thanks{Dedicated to Federico Menegazzo on the occasion of his 70th birthday}


\maketitle

\begin{abstract}
 We deal with the following conjecture. If $w$ is a group word and $G$ is a finite group in which any nilpotent subgroup generated by $w$-values has exponent dividing $e$, then the exponent of the verbal subgroup $w(G)$ is bounded in terms of $e$ and $w$ only. We show that this is true in the case where $w$ is either the $n$th Engel word or the word  $[x^n,y_1,y_2,\dots,y_k]$ (Theorem A). Further, we show that for any positive integer $e$ there exists a
number $k=k(e)$ such that if  $w$ is a word and  $G$ is a finite group
in which any nilpotent subgroup generated by products of $k$ values of
the word $w$ has exponent dividing $e$, then the exponent of the verbal
subgroup $w(G)$ is bounded in terms of $e$ and $w$ only (Theorem B).
\end{abstract}



\section{Introduction}
If $w$ is a group word in variables $x_1, x_2, \ldots, x_m$ we think of it as a function defined on any given group $G$. The subgroup of $G$ generated by the values of $w$ is called the verbal subgroup of $G$ corresponding to the word $w$. This will be denoted $w(G)$. The study of verbal subgroups of groups is a classical topic of group theory. It dates back to the theory of varieties of groups and the work of P. Hall.

A number of outstanding results about words in finite groups have been obtained in recent years. In this context we mention Shalev's theorem that for any nontrivial group word $w$, every element of every sufficiently large finite simple group is a product of at most three $w$-values \cite{shal}, and the proof by Liebeck, O'Brien, Shalev and Tiep \cite{lost} of Ore's conjecture: Every element of a finite simple group is a commutator. Another significant result is that of Nikolov and Segal that if $G$ is an $m$-generated finite group, then every element of $G'$ is a product of $m$-boundedly many commutators \cite{nisegal}. 

It was shown in \cite{S} that if $w$ is a multilinear commutator and $G$ a finite group in which any nilpotent subgroup generated by $w$-values has exponent dividing $e$, then the exponent of the verbal subgroup $w(G)$ is bounded in terms of $e$ and $w$ only.

Recall that a group has exponent $e$ if $x^e=1$ for all $x\in G$ and $e$ is the least positive integer with that 
property. Multilinear commutators (outer commutator words) are words which are obtained by nesting commutators, 
but using always different indeterminates. For example, the word $[[x_1,x_2],[x_3,x_4,x_5],x_6]$ is a multilinear 
commutator. On the other hand, many important words are not multilinear commutators. In particular, 
the $n$th Engel word, defined inductively by $$[x,_1\,y]=[x,y]\;\;\mathrm{ and }\;\; [x,_n\,y]=[[x,_{n-1}\,y],y],$$ is 
not a multilinear commutator when $n\geq 2$.

In view of the aforementioned result the following conjecture seems plausible.
\medskip

{\sc Conjecture.} 
{\it Let $w$ be any word and $G$ a finite group in which any nilpotent subgroup generated by $w$-values has exponent dividing $e$. Then the exponent of the verbal subgroup $w(G)$ is bounded in terms of $e$ and $w$ only.}
\medskip

Recall that a commutator word is a word $w$ such that $w(G) = 1$ for all abelian groups $G;$ this is equivalent to the requirement that for each variable $x$
 appearing in $w$ the sum of all exponents in the occurrences of $x$ in $w$ is zero. It is easy to see that for non-commutator words the above conjecture holds.

 The result in \cite{S} and our first result provide additional evidence in favor of this conjecture.
\medskip

{\sc Theorem A.} {\it Let $w$ be the $n$th Engel word or the word $[x^n,y_1,y_2,\dots,y_k]$. Assume that $G$ is a finite group in which any nilpotent subgroup generated by $w$-values has exponent dividing $e$. Then the exponent of the verbal subgroup $w(G)$ is bounded in terms of $e$ and $w$ only.}
\medskip

So far we have been unable to prove the conjecture for arbitrary words $w$. Yet, we obtained a result in the same direction that deals with arbitrary words. If $w$ is a word in $m$ variables, we define inductively $w_j$ as follows: $w_1=w$ and $w_j=w_{j-1}\cdot w$, where the variables appearing in $w_{j-1}$ and $w$ are disjoint, so that $w_j$ is a word in $mj$ variables. For instance, if $w=[x,y,y]$, then $w_2=[x_1, y_1,y_1][x_2,y_2,y_2]$. We remark that for any positive integer $j$, any group $G$ and any word $w$ we have $w(G)=w_j(G)$. 
\medskip

{\sc Theorem B.} {\it Given a positive integer $e$, there exists a number $k=k(e)$ such
that if $w$ is a word and $G$ is a finite group in which any nilpotent
subgroup generated by $w_k$-values has exponent dividing $e$, then the
exponent of the verbal subgroup $w(G)$ is bounded in terms of $e$ and
$w$ only.}
\medskip

We emphasize that the hypothesis of Theorem B does not imply that the word $w$ has bounded width in $G$. The reader can consult \cite{segal} for questions related to the width of a word in a finite group. It is easy to see that since we make no assumptions on the number of generators of $G$, the word $w$ can have arbitrarily large width.

We make no attempts to write down explicit estimates for the exponents of $w(G)$ in our results. 
 The next section is devoted to the proof of Theorem A. The proof of Theorem B will be given in Section 3. Both proofs depend on the classification of finite simple groups and on Zelmanov's solution of the restricted Burnside problem \cite{ze1,ze2}.

\section{Theorem A}

Throughout the paper we use the expression ``\{a,b, \ldots\}-bounded" to mean ``bounded from above by some function depending only on a,b \ldots". We say that a word $w$ has the property (bb) if there exists a constant $b$ depending on $w$ such that every finite group $H$ satisfying the law $w\equiv 1$ is an extension of a group of exponent dividing $b$ by a group which is nilpotent of class at most $b$. We note that the Engel word has the property (bb) by the main theorem in \cite{BM}; the fact that the word $w=[x^n,y_1,y_2,\dots,y_k]$  has the property (bb) will be proved shortly. As usual, $Z_i(H)$ and $\gamma_i(H)$  denote the $i$th term of the upper and lower central series of a group $H$, respectively. We use $\gamma_\infty(H)$ to denote the nilpotent residual of $H$, that is $\gamma_\infty(H)=\cap_{i\ge1}\gamma_i(H)$.

The next lemma is Lemma 2.2 of \cite{CS}. In the case where $k=1$ this is a well-known result, due to Mann \cite{M}.
\begin{lem}\label{mann+} If $G$ is a finite group such that $G/Z_k(G)$ has exponent $m$, then $\gamma_{k+1}(G)$ has $\{k,m\}$-bounded exponent.
 \end{lem}
The fact that the word $w=[x^n,y_1,y_2,\dots,y_k]$  has the property (bb) is now straightforward.
\begin{lem}\label{gamma} The word  $w=[x^n,y_1,y_2,\dots,y_k]$ has the property (bb). In particular, if $H$ is a finite group satisfying the law $w\equiv1$, then $\gamma_{k+1}(H)$ has $\{k,n\}$-bounded exponent.
   \end{lem}
\begin{proof} We have $H^n\le Z_{k}(H)$ and therefore $H/Z_k(H)$ has exponent dividing $n$. Thus, by Lemma \ref{mann+}, $\gamma_{k+1}(H)$ has $\{k,n\}$-bounded exponent.
\end{proof}

If $H$ is a subgroup of $G$, we write $w_{j,G}(H)$ for the subgroup generated by all $w_j$-values of $G$ lying in $H$. We write $w_{G}(H)$ for $w_{1,G}(H)$.

The class of all finite groups in which every nilpotent subgroup generated by $w_k$-values has exponent dividing $e$ will be denoted by $X(k,e,w)$. We write $X(e,w)$ in place of $X(1,e,w)$. If our conjecture 
 holds, the exponents of all groups in $X(e,w)$ have a common bound. Our Theorem A says that this is so when $w$ is either the $n$th Engel word or the word $[x^n,y_1,y_2,\dots,y_k]$. Theorem B says that for any $w$ and $e$ there exists $k$ such that the exponents of all groups in $X(k,e,w)$ have a common bound.

It seems that the class $X(e,w)$ is not closed with respect to quotients. Thus, it will be convenient to replace the condition defining the class $X(e,w)$ with a more manageable condition. In what follows $Y(e,w)$ will denote the class of finite groups $G$ such that every $w$-value has order dividing $e$ and $w(P)$ has exponent dividing $e$ for every $p$-Sylow subgroup of $G$.
It is clear that the class $Y(e,w)$ is closed under taking quotients of its members.

The next lemma is a straightforward consequence of \cite[5.2.5]{Rob}.
\begin{lem}\label{rob} Let $w$ be a word with the property (bb) and $G$ a finite group satisfying the law $w\equiv1$. Assume that $G$ is generated by elements of orders dividing $r$. Then $G$ has $(r,w)$-bounded exponent.
\end{lem}

It is clear from the definitions that $X(e,w)\subseteq Y(e,w)$. In the case where $w$ has the property (bb) an ``almost converse" also  holds.

\begin{lem}\label{quotiens} Assume that $e\geq 1$ and the word $w$ has the property (bb). Then there exists an $(e,w)$-bounded integer $f$ such that $Y(e,w)\subseteq X(f,w)$.
\end{lem}

\begin{proof} Choose $G\in Y(e,w)$. Let $H$ be a nilpotent subgroup of $G$ generated by $w$-values. Since  $H/w(H)$ satisfies the hypotheses of
Lemma \ref{rob}, it follows that  $H/w(H)$ has $(e,w)$-bounded exponent.
Moreover, if $H=P_1\times\dots\times P_r$, where the $P_i$'s are the Sylow subgroups of $H$, we have $w(H)=w(P_1)\times\dots\times w(P_r)$. As $G\in  Y(e,w)$ the subgroup $w(P_i)$ has exponent dividing $e$ for every $i$.  It follows that $H$ has $(e,w)$-bounded exponent, as required.
\end{proof}

A tower of height $r$ in a finite group $G$ is a subgroup $T$ of the form $T=P_1\cdots P_r$ where
 \begin{itemize}
  \item[(1)] $P_i$ is a $p_i$-group for $i=1,\dots,r$, with $p_i$ a prime number.
  \item[(2)] $P_i$ normalizes $P_j$ for $i<j$.
  \item[(3)] $[P_i,P_{i-1}]=P_i$ for $i=2,\dots,r$.
 \end{itemize}
Let us denote by $\mathrm{Fit}(G)$ the Fitting subgroup of $G$ and by $F_i(G)$ the $i$th term of the upper Fitting series of $G$, defined recursively by $F_1(G)=\mathrm{Fit}(G)$ and $F_i(G)/F_{i-1}(G)=\mathrm{Fit}(G/F_{i-1}(G))$. If $G$ is a finite soluble group, the least number $h$ with the property that $F_h(G)=G$, is called the Fitting height of $G$. It is easy to see that the Fitting height of a tower $T=P_1\cdots P_r$ equals precisely $r$ (note that $P_2 \cdots P_r \leq \gamma_\infty( T)$).
Moreover a finite soluble group $G$ has Fitting height $h$ if and only if $h$ is the maximal number such that $G$ possesses a tower of height $h$ (see for instance Section 1 of \cite{T}).
A proof of the next lemma can be found in \cite[Lemma 2.2]{S}.
\begin{lem}\label{coprimeaction}  Let $A$ be a group of automorphisms of a finite group $G$ such that $(|A|,|G|)=1$. Suppose that $B$ is a normal subset of $A$ such that $A=\langle B\rangle$ and let $k\ge 1$ be an integer. Then $[G,A]$ is generated by the subgroups of the form $[G,b_1,\dots,b_k]$, where $b_1,\dots,b_k\in B$.
\end{lem}

\begin{lem}\label{towerEngel}
Let $w$ be the $n$th Engel word and $G\in Y(e,w)$. If $T=P_1\cdots P_r$ is a tower of height $r$ in $G$, then every $p$-Sylow subgroup of $P_2\cdots P_r$  has $(e,w)$-bounded exponent.
\end{lem}

\begin{proof}
We will first show that for each $i=2,\dots,r$ the subgroup $P_i$ is generated by $w$-values. Let $N=w_{G}(P_i)$. We wish to show that $N=P_i$. We put $\overline{P_{i-1}P_i}=P_{i-1}P_i/N$
and use the bar notation in the quotient. If $\bar a \in \bar P_{i-1}$ then $[\bar P_i,_n \bar a]=1$. Let $\bar A=\langle\bar a\rangle$. By Lemma \ref{coprimeaction} $[\bar P_i,\bar A]$
is generated by subgroups of the form $[\bar P_i,_n \bar a]$, whence $[\bar P_i,\bar A]=1$. This happens for every $\bar a\in\bar P_{i-1}$ and so $[\bar P_i,\bar P_{i-1}]=1$. As $P_i=[P_i,P_{i-1}]$, it follows that $\bar P_i=1$ and $P_i=N$, as required.

Now let $P$ be a $p$-Sylow subgroup of $P_2\cdots P_r$. We know that $P$ is 
generated by $w$-values, which are all of orders dividing $e$. Since $G \in  Y(e,w)$, it follows that the exponent of $w(P)$ divides $e$. Therefore the exponent of $w(P)$ is $(e,w)$-bounded. An application of Lemma \ref{rob} now completes the proof.
\end{proof}

To prove an analogue of Lemma \ref{towerEngel} for the word $w=[x^n,y_1,y_2,\dots,y_k]$ we need a further technical lemma.
\begin{lem}\label{PQ}  Assume that $G=PQ$, where $P$ is a $p$-group and $Q$ is a $q$-group for different primes $p$ and $q$. Assume further that $P$ is normal in $G$ and $G\in Y(e,w)$, where $w$ is the word $w=[x^n,y_1,y_2,\dots,y_k]$. Then the exponent of $\gamma_{k+1}(G)$ is $(e,w)$-bounded.
 \end{lem}

\begin{proof}
 Let $M=w_{G}(P)$. By Lemma \ref{quotiens} the group $G$ and all of its quotients belong to $X(f,w)$ for some $(e,w)$-bounded integer $f$. As $M$ is nilpotent, its exponent divides $f$. It is clear that $[P,_{k}\,G^n]\leq M$ and so the exponent of $[P,_{k}\,G^n]$ divides $f$, too. Passing to the quotient $G/[P,_{k}\,G^n]$ we can assume that $[P,_{k}\,G^n]=1$ and so $G^n\cap P\le Z_{k}(G^n)$. The group $G^n/G^n\cap P$  is a $q$-group and so it is nilpotent. Therefore $G^n$ is nilpotent and we deduce that $[G^n,_k\,G]$ is of bounded exponent because it is a nilpotent subgroup generated by $w$-values and $G\in X(f,w)$. We can now pass to the quotient $G/[G^n,_k\,G]$ and without loss of generality assume that $w(G)=1$. In view of Lemma \ref{gamma} the result follows.
\end{proof}

\begin{lem}\label{towerw} Let $w=[x^n,y_1,y_2,\dots,y_k]$ and let $G\in Y(e,w)$. If $T=P_1\cdots P_r$ is a tower of height $r$ in $G$, then every $p$-Sylow subgroup of $P_2\cdots P_r$  has $(e,w)$-bounded exponent.
\end{lem}
\begin{proof}
We will first show that $P_i$ is generated by elements of bounded orders for each $i=2,\dots,r$. As $P_i=[P_i,P_{i-1}]$ it follows from Lemma \ref{coprimeaction} that $P_i$ is generated by subgroups of the form $[P_i,b_1,\dots,b_k]$ with $b_1,\dots,b_k\in P_{i-1}$. As the group $P_iP_{i-1}$ satisfies the hypotheses of Lemma \ref{PQ}, it follows that $\gamma_{k+1}(P_{i-1}P_i)$ has $(e,w)$-bounded exponent and this proves that $P_i$ is generated by elements of bounded orders for each $i=2,\dots,r$. Now let $P$ be a $p$-Sylow subgroup of $P_2\cdots P_r$. Since $G\in Y(e,w)$, it follows
that the exponent of $w(P)$ divides $e$. Since $P$ is generated by elements of bounded orders, the result follows from Lemma \ref{rob}.
\end{proof}

\begin{lem}\label{Fit} Let $w$ be either the $n$th Engel word or the word $[x^n,y_1,y_2,\dots,y_k]$. The Fitting height of any soluble group in $Y(e,w)$ is $(e,w)$-bounded.
\end{lem}
\begin{proof} Let $G$ be a soluble group in $Y(e,w)$ and choose a tower 
$T=P_1\cdots P_r$ in $G$ of height precisely the Fitting height of $G$. Of course, it is enough to prove that the Fitting height of the subgroup $P_2\cdots P_r$ is $(e,w)$-bounded. By Lemmas \ref{towerEngel} and \ref{towerw} every $p$-Sylow subgroup of $P_2\cdots P_r$ has $(e,w)$-bounded exponent. Therefore $P_2\cdots P_r$ is a soluble group of $(e,w)$-bounded exponent. According to the Hall-Higman theory \cite{HH} the Fitting height of $P_2\cdots P_r$ is bounded in terms of the exponent. This completes the proof.
 \end{proof}

 We need another technical result concerning the $n$th Engel word.
 \begin{lem}\label{engel} Let $w$ be the $n$th Engel word and let $T$ be a metanilpotent group. If $M$ is the subgroup generated by all $w$-values contained in the Fitting subgroup of $T$, then $T/M$ is nilpotent.
 \end{lem}
\begin{proof} Let $F$ be the Fitting subgroup of $T$ and let $g,a\in T$. As $T/F$ is nilpotent, $[g,_r\,a]\in F$ for some positive integer $r$. Then $[g,_{r+n}\,a]\in M$. Therefore every element of $T/M$ is left  Engel. It follows that $T/M$ is nilpotent (see for instance Proposition 12.3.3 of \cite{Rob}), as required.
 \end{proof}

\begin{prop}\label{solv} Let $w$ be the $n$th Engel word or the word $[x^n,y_1,y_2,\dots,y_k]$. Assume that $G$ is a soluble group in $Y(e,w)$. Then the exponent of $w(G)$ is $(e,w)$-bounded.
\end{prop}

\begin{proof}
The proof is by induction on the Fitting height $h$ of $G,$ which is $(e,w)$-bounded by Lemma \ref{Fit}.
 If $h=1$, then $G$ is nilpotent and so the exponent of $w(G)$ divides $e$. 

Now assume that $h>1$ and let $F$ be the Fitting subgroup of $G$. The subgroup $M=w_{G}(F)$ is nilpotent, so it has $(e,w)$-bounded  exponent by Lemma \ref{quotiens}.

Consider the case in which $w$ is the $n$th Engel word. Let $T/F$ be the Fitting subgroup of $G/F$. Then $T$ is metanilpotent and by Lemma \ref{engel} $T/M$ is nilpotent. Therefore the Fitting height of $G/M$ is smaller  than that of $G$. By induction $w(G/M)$ has $(e,w)$-bounded exponent. Hence, the exponent of $w(G)$ is $(e,w)$-bounded as well.

Now let  $w=[x^n,y_1,y_2,\dots,y_k]$. Put $K=G^n$. In the quotient $\bar G=G/M$ we have $[\bar K,\bar F,_{k-1}\,\bar G]=1$, whence $\bar F\cap\bar K\le Z_k(\bar K)$. This implies that $\bar K$ has smaller Fitting height than $G$. By induction $w(\bar K)$ has $(e,w)$-bounded exponent. Since the exponent of $M$ is bounded as well, we deduce that the exponent of $w(K)M$ is bounded. Passing to the quotient over $w(K)M$, we may assume that $w(K)=1$. It follows that $K^n$ is nilpotent. By Lemma \ref{quotiens} $G\in X(f,w)$ for some $(e,w)$-bounded integer $f$. Therefore the subgroup $w_G(K^n)$ has exponent dividing $f$. Passing to the quotient over $w_G(K^n)$, we can assume that $[K^n,_k\,G]=1$. This implies that $K^n\le Z_k(G)$, that is $G^{n^2}\le Z_k(G)$. Thus $G/Z_k(G)$ has bounded exponent. By Lemma \ref{mann+} $\gamma_{k+1}(G)$ has bounded exponent, and as $w(G)\le\gamma_{k+1}(G)$ the result follows.
\end{proof}

Following the terminology used by Hall and Highman in \cite{HH}, we say that a group $G$ is monolithic if it has a unique minimal normal subgroup which is
a non-abelian simple group. Nowadays very often such groups are called almost simple.

\begin{lem}\label{L} Let $L$ be a residually monolithic group satisfying a law $u\equiv 1$. Then $L$ has $u$-bounded exponent.
 \end{lem}
 
\begin{proof} Observe that every monolithic group is isomorphic to a subgroup of $\mathrm{Aut}(S)$, where $S$ is the non-abelian finite simple group isomorphic to the unique minimal normal subgroup of the group. A result of Jones \cite{J} says that any infinite family of finite simple groups generates the variety of all groups, so there are only finitely many finite simple groups satisfying the law $u\equiv 1$. This implies the existence of a bound (depending only on $u$) for the exponent of $\mathrm{Aut}(S)$, where $S$ ranges through such simple groups. Thus, $L$ has bounded exponent.
\end{proof}

Let $G$ be a finite group and $r$ a positive integer. As in \cite{S2} we will associate with $G$ a triple of numerical parameters $n_r(G)=(\lambda,\nu,\mu)$ where the parameters $\lambda,\mu,\nu$ are defined as follows.
 
We recall that the  $n$th derived word $\delta_n$  is defined recursively by: 
$\delta_1(x_1,x_2)=[x_1,x_2]$ and $\delta_i(x_1,\dots,x_{2^i})=[\delta_{i-1}(x_1,\dots,x_{2^{i-1}}),\delta_{i-1}(x_{2^{i-1}+1},\dots,x_{2^i})]$ for all $i>1$. Let $X_n(G)$ be the set of $\delta_n$-values. If $G$ is of odd order, we set $\lambda=\mu=\nu=0$. Suppose that $G$ is of even order and choose a $2$-Sylow subgroup $P$ in $G$. If the derived length $\mathrm{dl}(P)$ of $P$ is at most $r+1$ we define $\lambda=\mathrm{dl}(P)-1$. Put $\mu=2$ if $X_{\lambda}(P)$ contains elements of order greater than two and $\mu=1$ otherwise. We let $\nu=\mu$ if $X_{\lambda}(P)\not\subseteq Z(P)$ and $\nu=0$ if $X_{\lambda}(P)\subseteq Z(P)$.
 
If the derived length of $P$ is at least $r+2$ we define $\lambda=r$. Then $\mu$ will denote the number with the property that $2^{\mu}$ is the maximum of orders of elements in $X_{r}(P)$. Finally, let $2^{\nu}$ be the maximum of orders of commutators $[a,b]$, where $b\in P$ and $a=c^{\mu-1}$ for some element $c$ of maximal order in $X_{r}(P)$.

We remark that our definition of $\mu$ is slightly different from that given in \cite{S2}. We believe that the definition in \cite{S2} was somewhat inaccurate while the present definition corrects the error. This minor correction does not require any changes in the subsequent arguments. In particular, the proof of Proposition \ref{G/L} below remains unaffected.

The set of all possible triples $n_r(G)$ is naturally endowed with the lexicographic order. Moreover, if $N$ is a normal subgroup of $G$, then $n_r(G/N) \le n_r(G)$. As we will use induction on $n_r(G)$, we need to show that for some suitable $r$, depending only on $e$ and $w$, there are only finitely many triples $n_r(G)$ associated with groups $G$ in $Y(e,w)$.

\begin{lem}\label{nr} Let $w$ be a word with the property (bb) and let $e\ge 1$. 
There exist $(e,w)$-bounded numbers $r,\lambda_0,\mu_0,\nu_0$ such that $n_r(G)\le(\lambda_0,\mu_0,\nu_0)$ for every group $G\in Y(e,w)$.
\end{lem}

\begin{proof} Choose $G\in Y(e,w)$. Let $P$ be a $2$-Sylow subgroup of $G$. We will show that for a suitable $(e,w)$-bounded $r$ the exponent of $P^{(r)}$ is  $(e,w)$-bounded, where $P^{(r)}$ is the $r$th term of the derived series of $P$. By the hypothesis $w(P)$ has exponent dividing $e$. Since $w$ has the property (bb), it follows that $P/w(P)$ is an extension of a group of bounded exponent by a group of bounded nilpotency class. Therefore there exists an $(e,w)$-bounded number $r$ such  that $P^{(r)}$ is of bounded exponent. If $n_r(G)=(\lambda,\mu,\nu)$, the definitions imply that $\lambda\le r$ and $\mu,\nu$ are bounded by the exponent of $P^{(r)}$.
\end{proof}

For the proof of the next proposition see \cite{S2}.
\begin{prop}\label{G/L} Let $r\ge 1$ and let $G$ be a group of even order such that $G$ has no nontrivial normal soluble subgroups. Then $G$ possesses a normal subgroup $L$ such that $L$ is residually monolithic and $n_r(G/L)<n_r(G)$.
\end{prop}

An automorphism $x$ of a group $G$ is called a nil-automorphism if for every $g\in G$ there exists an integer $n=n(g)$ such that $[g,_nx]=1$, where the commutator is taken in the holomorph of $G$. We will  need the following result, which is an immediate consequence of \cite[Theorem A]{CP}.
\begin{prop}\label{unip} Any group of nil-automorphisms of a finite group is nilpotent.
\end{prop}

We are now in a position to complete the proof of Theorem A.

\begin{proof}
[Proof of Theorem A]
 We will prove a formally stronger statement that the exponent of $w(G)$ is $(e,w)$-bounded for any $G\in Y(e,w)$. This will be sufficient for our purposes as we know that  $X(e,w)\subseteq Y(e,w)$. Throughout, we use the fact that any quotient of a group in $Y(e,w)$ belongs to $Y(e,w)$. 

Let $r$ be as in Lemma \ref{nr}. As $r$ and $n_r(G)$ are $(e,w)$-bounded, we can use induction on $n_r(G)$. If $n_r(G)=(0,0,0)$, then $G$ has odd order and so by the Feit-Thompson theorem \cite{FT} $G$ is soluble. In this case the result holds by Proposition \ref{solv}. So we may assume that $n_r(G)>(0,0,0)$.

First suppose that $G$ has no nontrivial normal soluble  subgroups. By Proposition \ref{G/L} there exists a normal subgroup $L$ in $G$ such that $L$ is residually monolithic and $n_r(G/L)<n_r(G)$. By induction the exponent of $w(G)L/L$ is $(e,w)$-bounded. 
By Lemma \ref{L}, applied with $u=w^e$, the exponent of $L$ is also $(e,w)$-bounded so the result follows.

Now let us drop the assumption that $G$ has no nontrivial normal   soluble  subgroups and let $S$ be the soluble radical of $G$. Proposition \ref{solv} shows that the exponent of $w(S)$ is bounded. Thus, we pass to the quotient $G/w(S)$ and without loss of generality assume that $w(S)=1$. As $w$ has the property (bb), $S$ has a characteristic subgroup $N$ of bounded exponent such that $S/N$ is nilpotent of bounded class. Passing to the quotient $G/N$ we may assume that $S$ is nilpotent of bounded class. Therefore $w_{G}(S)$ is a nilpotent subgroup of $G$ generated by $w$-values. By Lemma \ref{quotiens} $w_{G}(S)$ has $(e,w)$-bounded exponent. Passing to the quotient $G/w_{G}(S)$ we may assume that  $w_{G}(S)=1$. 

Suppose that $w$ is the $n$th Engel word and let $a\in G$. Then $[S,_n a]\le w_{G}(S)=1$. Thus, every element of $G$ induces a nil-automorphism of $S$ and so by Proposition \ref{unip} $G/C_G(S)$ is nilpotent. Let $H=SC_G(S)$. As $G/S$ has no nontrivial normal soluble subgroups, the same holds for $H/S$. We have $n_r(H/S)\le n_r(H)\le n_r(G)$. Since we know that the theorem holds for groups  without nontrivial  normal soluble subgroups, we conclude that $w(H)S/S$ is of bounded exponent. As $S$ is nilpotent of bounded class we deduce that $S\le Z_c(H)$ for some bounded integer $c$, whence $w(H)/Z_c(w(H))$ has bounded exponent. Lemma \ref{mann+} applied to the group $w(H)$ tells us that $\gamma_{c+1}(w(H))$ has bounded exponent. We pass to the quotient $G/\gamma_{c+1}(w(H))$ and without loss of generality assume that $\gamma_{c+1}(w(H))=1$. Now $w(H)$ is a nilpotent subgroup of $G$ generated by $w$-values. By Lemma \ref{quotiens} $w(H)$ has $(e,w)$-bounded exponent and we may assume that $w(H)=1$. In this case $H$ is an $n$-Engel group. By Zorn's theorem \cite{Z} $H$ is nilpotent. Moreover $G/H$ is also nilpotent because $C_G(S)\le H$ and so $G$ is metanilpotent. Now the result follows from Proposition \ref{solv}.

From now on we assume that $w=[x^n,y_1,y_2,\dots,y_k]$. Since $G/S$ has no nontrivial soluble normal subgroups and $n_r(G/S)\le n_r(G)$, it follows that $w(G)S/S$ has bounded exponent. Put $K=G^n$. Recall that $S$ contains no nontrivial $w$-values. Therefore $[K,S,_{k-1}\,G]=1$ and so $S\cap K\le Z_k(K)$. Combining the fact that $w(K)S/S$ has bounded exponent with Lemma \ref{mann+} we deduce that $\gamma_{k+1}(w(K))$ has bounded exponent. Passing to the quotient over $\gamma_{k+1}(w(K))$ we may assume that $\gamma_{k+1}(w(K))=1$. This means that $w(K)$ is nilpotent, in which case it has bounded exponent since $G \in Y(e,w)\subseteq X(f,w)$ for a bounded integer $f$.
Again, we may assume that $w(K)=1$. Since $w(G)\le K$, we have $w(w(G))=1$ and now the conclusion is immediate from Lemma \ref{rob}.
\end{proof}

\section{Theorem B}

\begin{lem}\label{we} Let $G\in X(k,e,w)$. If $N$ is a normal subgroup of $G$ which contains no nontrivial $w_e$-values of $G$, then $[N,w(G)]=1$.
\end{lem}

\begin{proof} Let $x\in N$ and let $y$ be a $w$-value. As $G\in X(k,e,w)$, the order of $y$ divides $e$. Write $[y,x]=y^{-1}y^x=y^{e-1}y^x$. We see that $[y,x]$ is a $w_e$-value since $y^{e-1}$ is a $w_{e-1}$-value and $y^x$ is a $w$-value. By the hypothesis $N$ contains no nontrivial $w_e$-values so $[y,x]=1$ and the result follows.
\end{proof}

The following proposition is immediate from Theorem 2 in \cite{Se}.

\begin{prop}\label{Seg} There exists a function $s(d)$ such that if $G$ is a finite soluble group generated by $g_1,\dots,g_d$, then every element of the derived group is equal to a product of $s(d)$ commutators of the form $[x,g]$, where $x\in G$ and $g\in\{g_1,\dots,g_d\}$.
\end{prop}

The next proposition, whose proof is based on techniques developed in \cite{FGG}, is taken from \cite{S2}.

\begin{prop}\label{Fla} Let $G$ be a finite group and $x\in G$. Suppose  that every $4$ conjugates of $x$ generate a soluble subgroup of $G$ of Fitting height at most $h$. Then $x\in F_h(G)$.
\end{prop}

\begin{lem}\label{Fit2} Given  an integer  $e$, there exist two integers $h_0$ and $k_0$ such that if $w$ is a word and $H$ is a normal subgroup of a group $G\in X(k,e,w)$ with $k\ge k_0$
such that $w_{e,G}(H)$ is soluble, then the Fitting height of $w_{e,G}(H)$ is at most $h_0$.
\end{lem}

\begin{proof} Set $k_0=s(4)e^2$ and choose $G\in X(k,e,w)$ with $k\ge k_0$.
Suppose that $H$ is a normal subgroup of $G$ such that $w_{e,G}(H)$ is soluble and $g\in H$ is a $w_e$-value. Consider four arbitrary conjugates $g_1,g_2,g_3,g_4$ of $g$ and let $T=\langle g_1,g_2,g_3,g_4\rangle$. As $G\in X(k,e,w)$ the orders of $g_i$ divide $e$. For every $x\in G$ we have 
 $[x,g_i]=g_i^{-x}g_i=g_i^{(e-1)x}g_i$. It is clear that this is a $w_{e^2}$-value. By Proposition \ref{Seg} every element of $T'$ is the product of $s(4)$ commutators of the form $[x,g_i]$ for some $i=1,\dots,4$. Therefore every element of $T'$ is a $w_{e^2s(4)}$-value. As $k\ge s(4)e^2$ and  $G\in X(k,e,w)$ it follows that the derived subgroup $T'$ of $T$ has exponent dividing $e$. On the other hand, $T/T'$ is abelian and generated by elements of order dividing $e$. Hence $T$ has exponent dividing $e^2$. According to the Hall-Higman theory \cite{HH} it follows that $T$ has Fitting height bounded by a constant $h_0$ (which depends on $e$ only). Thus by Proposition \ref{Fla} $g\in F_{h_0}(H)$. As this holds for every $w_e$-value $g$ contained in $H$, it follows that $w_{e,G}(H)\le F_{h_0}(H)$, so the Fitting height of $w_{e,G}(H)$ is at most $h_0$.
\end{proof}

\begin{lem}\label{quo2} Assume that $G\in X(k,e,w)$ and let $t$ be an integer such that $1\le t\le k/e$. If $N$ is a normal subgroup of $G$, we have $G/N\in X(t,e,w)$.
\end{lem}
\begin{proof} We put $\bar G=G/N$ and use the bar notation in the quotient. Let $\bar H$ be a nilpotent subgroup of $\bar G$ generated by $w_t$-values $\bar a_1,\dots,\bar a_s$. As $\bar H$ is nilpotent and each $\bar a_i$  has order dividing $e$, a $p$-Sylow subgroup of $\bar H$ is of the form $\bar P=\langle \bar a_1^m,\dots,\bar a_s^m\rangle$, where $p$ is a prime dividing $e$ and  $e=p^\alpha m$, with $(m,p)=1$. Let $P$ be a $p$-Sylow subgroup of $H$ (so that $ PN/N=\bar P$) and let $X$ be the set of the values of the word $w_t^m$ in $G$. Then $X$ is a normal subset of $H$ consisting of $p$-elements. By \cite[Lemma 2.1]{AFS} $\bar P=\langle \overline{P\cap X}\rangle$. As $tm\le te\le k$, the subgroup $\langle P\cap X\rangle$ is a nilpotent subgroup of $G$ generated by $w_k$ values. Therefore $\langle P\cap X\rangle$ has exponent dividing $e$, whence also $\bar P$ has exponent dividing $e$. Thus, $\bar H$ has exponent dividing $e$, as required.
\end{proof}

In the next lemma we refine arguments used in Lemma \ref{quo2} and Proposition 3.8 of \cite{CS}.

\begin{lem}\label{prel} Assume that  $G\in X(k,e,w)$ is a group without nontrivial normal soluble subgroups, where $k\ge e^3$. Suppose that $s$ is a positive integer such that $k\ge es$.  Then $G$ possesses a normal subgroup $L$ with the property that $L$ is residually monolithic and  $G/L$ is a subdirect product of groups $G_i\in X(s,e/p_{i},w)$, where $p_i$ is a suitable prime divisor of $e$.
 \end{lem}
\begin{proof} Choose a minimal normal subgroup $M$ in $G$. Write $M=S_1\times\dots\times S_t$, where the $S_i$'s are isomorphic nonabelian simple subgroups and let $L_M$ be the kernel of the natural permutation action of $G$ on the set $\{S_1,\dots,S_t\}$. We put $\bar G=G/L_M$ and use the bar notation in the quotient. Suppose first that $w_{e,G}(M)\neq 1$ Let $b$ be a nontrivial $w_e$-value lying in $M$ and write $b=b_1\dots b_t$, where $b_i\in S_i$. Choose $j$ such that that $b_j\ne 1$. Since $G$ acts on the set $\{S_1,\dots,S_t\}$ transitively, without loss of generality we can assume that $j=1$. Let now $y\in S_1$ be an element which does not centralize $b_1$.
Then $z=[b,y]=b^{-1}b^y=b_1^{-1}b_1^y$ is a non-trivial $w_{e^2}$-value lying in $S_1$. Since $k\ge e^3$ and $G\in X(k,e,w)$, the order of $z$ divides $e$.

Let $p$ be a prime dividing the order of $z$. Let $\bar H$ be a nilpotent subgroup of $\bar G$ generated by $w_{s}$-values $\bar a_1,\dots,\bar a_v$. The subgroup $\bar H$ is the direct product of its Sylow subgroups. Let $\bar R$ be a $r$-Sylow subgroup of $\bar H$. If $e=r^\alpha m$ with $(r,m)=1$, then $\bar R=\langle \bar a_1^m,\dots,\bar a_v^m\rangle$. Denote by $X$ the set of all values of the word $w_{s}^m$ in $G$. 
As $s\le k$, it follows that $X$ is a normal subset of $G$ consisting of $r$-elements of order dividing $e$. 

Let $R$ be an $r$-subgroup of $G$ such that $RL_M/L_M=\bar R$ and let $T$ be an $r$-Sylow subgroup of $G$ containing $R$. Then $\langle T\cap X\rangle$ is a nilpotent subgroup of $G$ generated by $w_k$-values and so it has exponent dividing $e$ because $G\in X(k,e,w)$. By Lemma 2.1 of \cite{AFS} $\langle \bar T\cap \bar X\rangle=\langle \overline{T\cap X}\rangle$. So $\bar R=\langle \bar a_1^m,\dots,\bar a_v^m\rangle\le\langle \overline{T\cap X}\rangle$ has exponent dividing $e$. If $r\ne p$, then the maximum power of $r$ dividing $e$ is the same as the maximum power of $r$ dividing $e/p$, so $\bar R$ has exponent dividing $e/p$. Suppose now that $r=p$ and assume by contradiction that there exists $\bar x\in \bar R$ such that $\bar x^{e/p}\ne 1$. Recall that $e=p^\alpha m$. Hence $\bar x$ has order $p^\alpha$. Since $\bar x\in \bar R\le \langle \overline{T\cap X}\rangle$, we can choose  $x\in \langle T\cap X\rangle$ such that $xL_M=\bar x$. The element $x$ permutes regularly $p^\alpha$ subgroups among the $S_i$'s and we can assume that $S_1$ is one of those. As $e^2\le k$, the element $z$ is a $w_k$-value and therefore it has order dividing $e$. Set $z_1=z^m$ and notice that $z_1$ is a nontrivial $p$-element which is also a $w_k$-value, 
because $e^2m<e^3\le k$. It is clear that $T\cap M$ is a $p$-Sylow subgroup of $M$, so it is the direct product of suitable $p$-Sylow subgroups of $S_i$, one for each $i$. Thus we can conjugate $z$ by a suitable element of $S_1$ and assume that $z_1\in T$. As $z_1$ is a $w_{e^2}^m$-value, we see that $z_1\in T\cap X$. Therefore $z_1x\in\langle T\cap X\rangle$, which is a nilpotent subgroup of $G$ generated by $w_k$-values. Hence, $z_1x$ has order diving $e$. In particular, as $z_1x$ is a $p$-element, its order divides $p^\alpha$. But $$(z_1x)^{p^\alpha}=z_1z_1^{x^{-1}}z_1^{x^{-2}}\dots z_1^{x}$$ and as all different conjugates of $z_1$ in the above expression lie in different factors $S_i$'s, it follows that $(z_1x)^{p^\alpha}\neq1$, a contradiction. Thus, $\bar G\in X(s,e/p,w)$.

Now consider the case where our minimal normal subgroup $M$ has the property that $w_{e,G}(M)=1$. By Lemma \ref{we} we have $[w(G),M]=1$ and therefore $w(G)\le L_M$. In this case $\bar G\in X(e^2,1,w)\leq X(s,e/p_i,w)$ for any prime divisor $p_i$ of $e$.

Let $L$ be the intersection of all the subgroups $L_M$ where $M$ ranges through the minimal normal subgroups of $G$. Then $G/L$ is a subdirect product of the groups $G/L_M$, so we just need to show that $L$ is residually monolithic. If $N$ is the product of all minimal normal subgroups of $G$, it is clear that $N$ is the direct product of pairwise commuting simple groups $N_1,\dots, N_l$ and $L$ is the intersection of the normalizers of the $N_i$'s. The subgroup $L$ acts on each $N_i$ by conjugation and let $\rho_i:N\to \mathrm{Aut}(N_i)$ be the natural homomorphism. It is easy to see that the image of $\rho_i$ is monolithic. Since $G$ has no nontrivial normal soluble subgroups it follows that $C_G(N)=1$ and so $L$ embeds into the direct product of the groups $\rho_i(N)$ for $i=1,\dots,l$. The lemma follows.
 \end{proof}

Now we are ready to present a proof of Theorem B.
\begin{proof}
[Proof of Theorem B] 
We will use induction on $e$ to prove that there exist two numbers $k=k(e)$ and $E=E(e,w)$ such that if $G\in X(k,e,w)$, then $w(G)$ has exponent dividing $E$. If $e=1$, the result is trivial. Therefore we assume that $e\ge 2$ and there exist constants $k_1$ and $E_1$ such that the exponent of $w(D)$ divides $E_1$ whenever $D\in X(k_1,e/p,w)$ for some prime divisor $p$ of $e$. We let $k_0$ and $h_0$ have the same meaning as in  Lemma \ref{Fit2} and denote by $m$ the maximum of $\{e^2,k_0,k_1\}$. By Lemma \ref{L} there exists a constant $E_0$ depending only on $w$ and $e$ such that the exponent of any residually monolithic group satisfying the law $w^e\equiv1$ has exponent dividing $E_0$. 

We will show first that the constants $k_2=me$ and $E_2=E_0E_1$ have the property that if $G$ is a finite group in $X(k_2,e,w)$ without nontrivial normal soluble subgroups, then the exponent of the verbal subgroup $w(G)$ divides $E_2$. Thus, take $G\in X(k_2,e,w)$ and assume that $G$ has no nontrivial normal soluble subgroups. By Lemma \ref{prel} applied  with $s=m$, $G$ possesses a normal subgroup $L$ such that $L$ is residually monolithic and  $G/L$ is the subdirect product of groups $G_i\in X(m,e/p_{i},w)$, where $p_i$ is a suitable prime dividing $e$. By induction the exponent of $w(G_i)$ divides $E_1$ for every $i$ and we know that the exponent of $L$ divides $E_0$. It follows that indeed the exponent of $w(G)$ divides $E_2$.

We will now deal with the group $G\in X(k_2,e,w)$ but we drop the assumption that $G$ has no nontrivial normal soluble subgroups. The symbol $S(D)$ will stand for the soluble radical of a group $D$. Let $S=S(G)$ and $T=w_{e,G}(S)$. Since $k_2\ge k_0$, by Lemma \ref{Fit2}  
 it follows that the Fitting height $h=h(T)$ of $T$ is at most $h_0$. From now on we will be using $h$ as a second induction parameter.

Consider the case where $h=0$. This happens if and only if $T=1$. By Lemma \ref{we} $[S,w(G)]=1$. Therefore every nilpotent subgroup generated by $w_{k_2}$-values in $G/S$ is an image of a nilpotent subgroup generated by $w_{k_2}$-values in $G$. Hence it must have exponent dividing $e$. It follows that $G/S$ belongs to $X(k_2,e,w)$. Since $G/S$ has no nontrivial normal soluble subgroups, the above argument shows that $w(G/S)$ has exponent dividing $E_2$. Since $[S,w(G)]=1$, it follows that $w(G)/Z(w(G))$ has exponent dividing $E_2$. By Lemma \ref{mann+} we conclude that $w(G)'$ has $(e,w)$-bounded exponent. As $w(G)/w(G)'$ is abelian and generated by elements of order dividing $e$, it has exponent dividing $e$. Therefore the exponent of $w(G)$ is $(e,w)$-bounded. Thus, in the case where $h=0$ the result is established.

We will now assume that $h\ge1$. The induction hypothesis will be that there exist two numbers $k_3\geq k_0$ and $E_3$ such that if $D\in X(k_3,e,w)$ and $h(w_{e,D}(S(D)))\leq h-1$, then $w(D)$ has exponent dividing $E_3$. Put $k_4=ek_3$ and assume that our group $G$ belongs to $X(k_4,e,w)$.

Let $F$ be the Fitting subgroup of $T$ and $M=w_{e,G}(F)$. Since $M$ is nilpotent, it has  exponent dividing $e$. So we can consider $\bar G=G/M$ and it is enough to prove that $w(\bar G)$ has bounded exponent. By Lemma \ref{we}, $[\bar F,w(\bar G)]=1$. This implies in particular that $[\bar F,\bar T]=1$. Hence, $\bar F\le Z(\bar T)$ and so the Fitting height of $\bar T$ is smaller than $h$. Moreover, by Lemma \ref{quo2}, $\bar G\in X(k_3,e,w)$ and so by induction the exponent of $w(\bar G)$ divides $E_3$. Therefore the exponent of $w(G)$ divides $eE_3$ and this completes the induction step. As $h\le h_0$  is $e$-bounded,  the theorem follows.
\end{proof}

{\bf{Acknowledgements}}
This research was partially supported by MIUR (project Teoria dei gruppi 
e applicazioni). The third author also acknowledges the support of CNPq - Brazil.



\end{document}